\documentclass[a4paper,oneside,11pt]{article}

\usepackage{amsmath,amsfonts,amscd,amssymb}
\usepackage{longtable,geometry}
\usepackage[english]{babel}
\usepackage[active]{srcltx}
\usepackage[T1]{fontenc}
\usepackage{graphicx}
\usepackage{pstricks}
\usepackage{bbm}
\usepackage{mathtools}
\usepackage{hyperref}
\usepackage{scalerel}

\usepackage{MnSymbol}
\usepackage{stmaryrd}
\usepackage{nicefrac}
\usepackage{calrsfs}
\usepackage{enumitem}
\usepackage[mathscr]{eucal}

\usepackage{xcolor}
\usepackage{framed}
\usepackage{amsmath,amsfonts,amscd,amssymb}
\usepackage{chngcntr}
\usepackage{longtable,geometry}
\usepackage[english]{babel}
\usepackage[active]{srcltx}
\usepackage[T1]{fontenc}
\usepackage{graphicx}
\usepackage{pstricks}
\usepackage{bbm}
\usepackage{mathtools}
\usepackage{hyperref}
\usepackage{commath}
\usepackage{MnSymbol}
\usepackage{stmaryrd}
\usepackage{nicefrac}
\usepackage{calrsfs}
\usepackage{relsize}
\usepackage{cases}
\usepackage{enumitem}

\usepackage{xcolor}
\usepackage{framed}

\colorlet{shadecolor}{blue!15}

\usepackage{amsthm}
\newtheorem{theorem}{Theorem}

\newtheorem{corollary}[theorem]{Corollary}
\newtheorem{lemma}[theorem]{Lemma}

\newtheorem{definition}[theorem]{Definition}

\newtheorem{remark}[theorem]{Remark}

\newcommand{\calA}{\mathcal{A}}
\newcommand{\calB}{\mathcal{B}}
\newcommand{\calC}{\mathcal{C}}
\newcommand{\calD}{\mathcal{D}}
\newcommand{\calE}{\mathcal{E}}
\newcommand{\calF}{\mathcal{F}}
\newcommand{\calG}{\mathcal{G}}
\newcommand{\calH}{\mathcal{H}}

\newcommand{\calO}{\mathcal{O}}
\newcommand{\calP}{\mathcal{P}}

\newcommand{\bbE}{\mathbb{E}}

\newcommand{\bbN}{\mathbb{N}}

\newcommand{\bbP}{\mathbb{P}}

\newcommand{\bbR}{\mathbb{R}}

\newcommand{\bbZ}{\mathbb{Z}}

\newcommand{\ep}{\varepsilon}

\newcommand{\var}{\mathrm{Var}}
\newcommand{\infl}{\mathrm{Inf}}

\numberwithin{equation}{section}

\newcommand{\rk}[1]{\bgroup\color{red}%
  \par\medskip\hrule\smallskip%
  \noindent\textbf{#1}%
  \par\smallskip\hrule\medskip\egroup}

\counterwithout{equation}{section}

\title{Subcritical phase of $d$-dimensional Poisson-Boolean percolation and its vacant set}
\author{Hugo Duminil-Copin\thanks{Universit\'e de Gen\`eve} \thanks{Institut des Hautes \'Etudes Scientifiques} , Aran Raoufi\addtocounter{footnote}{-1}\footnotemark\ , Vincent Tassion\thanks{ETH Zurich}}
\date{\today}

\begin{document}
\maketitle
 
 \begin{abstract}
 We prove that the Poisson-Boolean percolation on $\bbR^d$ undergoes a sharp phase transition in any dimension under the assumption that the radius distribution has a $5d-3$ finite moment (in particular we do not assume that the distribution is bounded). More precisely, we prove that
 \begin{itemize}
 \item In the whole subcritical regime, the expected size of the cluster of the origin is finite, and furthermore we obtain bounds for the origin to be connected to distance $n$: when the radius distribution has a finite exponential moment, the probability decays exponentially fast in $n$, and when the radius distribution has heavy tails, the probability is equivalent to the probability that the origin is covered by a ball going to distance $n$. 
 \item In the supercritical regime, it is proved that the probability of the origin being connected to infinity satisfies a mean-field lower bound.
\end{itemize}
 The same proof carries on to conclude that the vacant set of Poisson-Boolean percolation on $\bbR^d$ undergoes a sharp phase transition. This paper belongs to a series of papers using the theory of randomized algorithms to prove sharpness of phase transitions, see \cite{DumRaoTas16a,DumRaoTas16b}.
\end{abstract}
 
\section{Introduction}

\paragraph{Definition of the model} Bernoulli percolation was introduced in \cite{BroHam57} by Broadbent and Hammersley  to model the diffusion of a liquid in a porous medium. Originally defined on a lattice, the model was later generalized to a number of other contexts. Of particular interest is the development of continuum percolation (see \cite{MeeRoy08} for a book on the subject), whose most classical example is provided by the Poisson-Boolean model (introduced by Gilbert \cite{Gil}). It is defined as follows. 

Fix a positive integer $d\ge2$ and let $\bbR^d$ be the $d$-dimensional Euclidean space endowed with the $\ell^2$ norm $\|\cdot\|$. For $r>0$ and $x\in\bbR^d$, set $\mathsf B_r^x:=\{y\in\bbR^d:\|y-x\|\le r\}$ and $\partial\mathsf B_r^x:=\{y\in \bbR^d:\|y-x\|=r\}$ for the ball and sphere of radius $r$ centered at $x$. When $x=0$, we simply write $\mathsf B_r$ and $\partial\mathsf B_r$. For a subset $\eta$ of $\bbR^d\times\bbR_+$, set 
$$\calO(\eta):=\bigcup_{(z,R)\in\eta} \mathsf B_R^z.$$

Let $ \mu$ be a measure on $\bbR_+$ (below, we use the notation $\mu[a,b]$ to refer to $\mu([a,b))$, where $a,b\in\bbR\cup\{\infty\}$) and $\lambda$ be a positive number. Let $\eta$ be a Poisson point process of intensity  $ \lambda \cdot dz\otimes \mu$, where $dz$ is the Lebesgue measure on $\bbR^d$. Write $\bbP_\lambda$
for the law of $\eta$ and $\bbE_\lambda$ for the expectation with respect to $\bbP_\lambda$. The random variable $\calO(\eta)$, where $\eta$ has law $\bbP_\lambda$ is called the {\em Poisson-Boolean percolation} of radius law $\mu$ and intensity $\lambda$. A natural hypothesis in the study of Poisson Boolean percolation is to assume $d$-th moment on the radius distribution:
\begin{equation}
  \label{eq:5}
  \int_0^\infty r^d d\mu(r) <\infty.
\end{equation}
Indeed, as observed by Hall \cite{Hal}, the condition~\eqref{eq:5} is necessary in order to avoid the entire space to
be almost surely covered, regardless of the intensity (as long as positive) of the Poisson point process.

\paragraph{Main result} Two points $x$ and $y$ of $\bbR^d$ are said to be {\em connected} (by $\eta$) if there exists a continuous path in $\calO(\eta)$ connecting $x$ to $y$. This event is denoted by $x\longleftrightarrow y$. For $X, Y \subset \bbR^d$, the event $\{X\longleftrightarrow Y\}$ denotes the existence of $x \in X$ and $y \in Y$ such that $x$ is connected to $y$ (when $X=\{x\}$, we simply write $x\longleftrightarrow Y$). 
Define, for every $r>0$, the two functions of $\lambda$
\begin{align*}
\theta_r(\lambda):=\bbP_\lambda[0\stackrel{}\longleftrightarrow \partial\mathsf B_r]\qquad\text{ and }\qquad\theta(\lambda):=\lim_{r\rightarrow\infty}\theta_r(\lambda).\end{align*}
We define the critical parameter $\lambda_c=\lambda_c(d)$ of the model by the formula 
$$\lambda_c:=\inf\{\lambda\ge 0:\theta(\lambda)>0\}.$$

Another critical parameter is often introduced to discuss Poisson-Boolean percolation. Define $\widetilde\lambda_c=\widetilde\lambda_c(d)$ by the formula
$$\widetilde\lambda_c:=\inf\{\lambda\ge0:\inf_{r>0}\bbP_\lambda[\mathsf B_r\stackrel{}{\longleftrightarrow}\partial\mathsf B_{2r}]>0\}.$$
This quantity is of great use as it enables one to initialize renormalization arguments; see e.g.~\cite{Gou1,Gou2} and references therein. As a consequence, a lot is known for Poisson-Boolean percolation with intensity $\lambda<\widetilde\lambda_c$.  We refer to Theorems~\ref{thm:2} and \ref{thm:3} and to \cite{MeeRoy08} for more details on the subject. 

Under the minimal assumption \eqref{eq:5}, Gouer\'e~\cite{Gou1} proved that $\lambda_c$ and $\widetilde{\lambda}_c$ are nontrivial. More precisely he proved $0<\widetilde\lambda_c\le\lambda_c <\infty$. The main result of this paper is the following.
\begin{theorem}[Sharpness for Poisson-Boolean percolation]\label{thm:main}
Fix $d\ge2$ and assume that 
\begin{equation}
\int_{\bbR_+} r^{5d-3}d\mu(r)<\infty.\label{eq:7}
\end{equation}
Then, we have that $\lambda_c=\widetilde\lambda_c$. Furthermore, there exists $c>0$ such that $\theta(\lambda)\ge c(\lambda-\lambda_c)$ for any $\lambda\ge \lambda_c$.
\end{theorem}

The case of bounded radius is already proved in \cite{zuev1985continuous, meester1994nonuniversality} (see also  \cite{MeeRoy08}), and we refer the reader to \cite{Zie16} for a new proof. Theorems stating sharpness of phase transitions for percolation models in general dimension $d$ were first proved in the eighties for Bernoulli percolation \cite{Men86,AizBar87} and the Ising model \cite{AizBarFer87}. The proofs of sharpness for these models (even alternative proofs like \cite{DumTas15}) harvested independence for Bernoulli percolation and special structures of the random-current representation of the Ising model. In particular, they were not applicable to other models of statistical mechanics. In recent years, new methods were developed to prove sharpness for a large variety of statistical physics models in two dimensions \cite{BolRio06,BefDum12,ahlberg2016sharpness}. These methods rely on general sharp threshold theorems for Boolean functions, but also on \emph{planar} properties of crossing events. In particular, the proofs use planarity in an essential way and seem impotent in higher dimensions.

Recently, the authors proved the sharpness of phase transition for random-cluster model \cite{DumRaoTas16a} and Voronoi percolation \cite{DumRaoTas16b} in arbitrary dimensions. The shared theme of the proofs is the use of randomized algorithms to prove differential inequalities for connection probabilities. Here, we adapt this theme in the context of Poisson-Boolean model. The major difference with the two previous papers is that we do not prove that connectivity probabilities decay exponentially fast in the distance below criticality. The reason is that this fact is simply not true in general for Poisson-Boolean percolation. Indeed, if the tail of the radius distribution is slower than exponential, then one can consider the event that a large ball covers two given points, an event which has a probability larger than exponential in the distance between the two points. 
Instead, we show that $\lambda_c=\widetilde\lambda_c$, without ever referring to exponential decay, by controlling the probability of connectivity functions for $ \lambda > \widetilde \lambda_c$, and by deriving a differential inequality which is valid in this regime.

\paragraph{Decay of $\theta_r(\lambda)$ when $\lambda<\widetilde\lambda_c$}
For standard percolation, sharpness of the phase transition refers to the exponential decay of connection probabilities in the subcritical regime. In Poisson-Boolean percolation with arbitrary radius law $\mu$, we mentioned above that one cannot expect such behavior to hold in full generality. In order to explain why the theorem above is still called ``sharpness'' in this article, we provide below some new results concerning the behavior of Poisson-Boolean percolation when $\lambda<\widetilde\lambda_c$. 

First, remark that for every $\lambda>0$, $\theta_r(\lambda)$ is always bounded from below by
\begin{equation}
  \label{eq:2}
  \phi_r(\lambda):=\mathbb P_\lambda[\exists (z,R)\in\eta\text{ such that } \mathsf B_R^z\text{ contains $0$ and intersects $\partial \mathsf B_r$}],
\end{equation}
whose decay may be arbitrarily slow. Nevertheless, one may expect the following phenomenology when $\lambda<\widetilde\lambda_c$: \begin{itemize}[noitemsep,nolistsep]
\item If $\mu[r,\infty]$ decays exponentially fast, then so does $\theta_r(\lambda)$ (but not necessarily at the same rate of exponential decay).\item Otherwise, the decay of $\theta_r(\lambda)$ is governed by $\phi_r(\lambda)$, in the sense that it is roughly equivalent to it.
\end{itemize}
The first item above is formalized by the following theorem.
\begin{theorem}
  \label{thm:2} If there exists $c>0$ such that $\mu[r,\infty]\le \exp(-cr)$ for every $r\ge1$, then, for every $\lambda<\widetilde\lambda_c$, there exists $c_\lambda>0$ such that for every $r\ge1$,
  \begin{equation*}
    \theta_r(\lambda)\le \exp(-c_\lambda r).
  \end{equation*}
\end{theorem}

Giving sense to the second item above is not easy in full generality, for instance when the law of $\mu$ is very irregular (one can imagine distributions $\mu$ that do not decay exponentially fast, but for which large range of radii are excluded). In Section~\ref{sec:4.3}, we give a general condition under which a precise description of $\theta_r(\lambda)$ can be obtained. To avoid introducing technical notation here, we only give two applications of the results proved in Section~\ref{sec:4.3}. We believe that these applications bring already a good idea of the general phenomenology. The proof mostly relies on new renormalization inequalities. We believe that these renormalization inequalities can be of great use to other percolation models. The theorem claims that the cheapest way for 0 to be connected to distance $r$ is if a single huge ball covers 0 and intersects the boundary of $\mathsf B_r$.

\begin{theorem}\label{thm:3}
Fix $d\ge2$. Let $\mu$ is of one of the two following cases:
\begin{enumerate}[noitemsep,nolistsep]
\item[C1]\label{item:1} There exists $c>0$ such that  $\mu[r,\infty]=1/r^{d+c}$ for every $r\ge 1$,
 \item[C2]\label{item:2} There exist $c>0$  and $0<\alpha<1$ such that $\mu[r,\infty]=\exp(-cr^{\alpha})$ for every $r\ge1$.\end{enumerate}
Then, for every $\lambda<\widetilde\lambda_c$,
 \begin{equation*}
    \lim_{r\rightarrow\infty}\frac{\theta_r(\lambda)}{\phi_r(\lambda)}=1.
  \end{equation*}
\end{theorem}
\paragraph{Vacant set of the Poisson-Boolean model} Another model of interest can also be studied using the same techniques. In this model, the connectivity of the points is given by continuous paths in the complement of $\calO(\eta)$. Write $x\stackrel{*}{\longleftrightarrow}y$ for the event that $x$ and $y$ are connected by a continuous path in $\bbR^d\setminus\calO(\eta)$, and $X\stackrel{*}{\longleftrightarrow}Y$ if there exist $x\in X $ and $y\in Y$ such that $x$ and $y$ are connected. 
For $\lambda$ and $r\ge0$, define 
\begin{align*}
\theta_r^*(\lambda):=\bbP_\lambda[0\stackrel{*}\longleftrightarrow \partial\mathsf B_r]\qquad\text{ and }\qquad\theta^*(\lambda)=\lim_{r\rightarrow\infty}\theta_r^*(\lambda).\end{align*}
We define the critical parameter $\lambda_c^*$  (see e.g.~\cite{Pen17} or~\cite{ATT17} for the fact that it is positive)  by the formula
\begin{align*}\lambda_c^*&:=\sup\{\lambda\ge 0:\theta^*(\lambda)>0\}.
\end{align*}
We have the following theorem.
\begin{theorem}[Sharpness of the vacant set of the Poisson-Boolean model]
Fix $d\ge2$ and assume that the radius distribution $\mu$ is compactly supported. Then, for all $\lambda < \lambda_c^*$, there exists $c_\lambda >0$ such that for every $r\ge1$,
$$ \theta_r^*(\lambda) \leq \exp (-c_\lambda r).$$
Furthermore, there exists $c>0$ such that for every $\lambda\ge\lambda_c^*$, 
$\theta^*(\lambda)\ge c(\lambda-\lambda_c^*).$
\end{theorem}
Since the proof follows the same lines as in Theorem~\ref{thm:main}, we omit it in the article and leave it as an exercise to the reader.

\paragraph{Strategy of the proof of Theorem~\ref{thm:main}}
Let us now turn to a brief description of the general strategy to prove our main theorem. Theorem \ref{thm:main} is a consequence of the following lemma.
\begin{lemma} \label{lem:mlem}
Assume the moment condition~\eqref{eq:7} on the radius distribution. Then, there exists a constant $c_1>0$ such that for every $r\geq 0$ and $\lambda\ge\widetilde\lambda_c$, 
\begin{equation} \label{eq:mlem}
{\theta'_r}(\lambda) ~\geq~ c_1 \,  {\frac{r}{{\Sigma_r(\lambda)}}} \, \theta_r(\lambda)( 1 - \theta_r(\lambda)),
\end{equation} 
where $\Sigma_r(\lambda):= \int_0^r \theta_s(\lambda)ds$.
\end{lemma}
The whole point of Section~\ref{sec:3} will be to prove Lemma~\ref{lem:mlem}. Before that, let us mention how it implies Theorem~\ref{thm:main}. 

\begin{proof}[Proof of Theorem~\ref{thm:main}]Fix $\lambda_0> \widetilde\lambda_c$. As in \cite[Lemma~3.1]{DumRaoTas16a}, Lemma~\ref{lem:mlem} implies that there exists $\lambda_1\in[\widetilde\lambda_c,\lambda_0]$ such that
\begin{itemize}[noitemsep]
\item for any $\lambda\ge\lambda_1$, $\theta(\lambda)\ge c(\lambda-\lambda_1)$.
\item for any $\lambda\in(\widetilde\lambda_c,\lambda_1)$, there exists $c_\lambda>0$ such that $\theta_r(\lambda)\le \exp(-c_\lambda r)$ for every $r\ge0$.
\end{itemize}
The two items imply that $\lambda_1=\lambda_c$. Yet, the second item implies that $\lambda_1\le \widetilde\lambda_c$, since clearly exponential decay would imply that for $\lambda\in(\widetilde\lambda_c,\lambda_1)$,
$$\lim_{r\rightarrow\infty}\bbP_\lambda[\mathsf B_r\stackrel{}{\longleftrightarrow}\partial\mathsf B_{2r}]=0.$$
 Since $\lambda_1 = \lambda_c\ge \widetilde\lambda_c$, we deduce that $\lambda_1=\lambda_c=\widetilde\lambda_c$ and the proof is finished.
 \end{proof}
\begin{remark} Note that we did not deduce anything from Lemma~\ref{lem:mlem} about exponential decay since eventually $\lambda_1=\widetilde\lambda_c$. It is therefore not contradictory with the case in which $\mu[r,\infty]$ does not decay exponentially fast.
 \end{remark}
The proof of Lemma~\ref{lem:mlem} relies on the OSSS inequality,  first proved in \cite{OSSS}, connecting randomized algorithms and influences in a product space. Let us briefly describe this inequality.
Let $ I$ be a finite set of {\em coordinates}, and let $\Omega =\prod_{i \in I} \Omega_i$ be a product space endowed with product measure $\pi =\otimes_{i \in I} \pi_i$. 
An {\em algorithm $\mathsf T$} determining $f: \Omega \rightarrow \{0,1\}$ takes a configuration $\omega=(\omega_i)_{i\in I}\in\Omega$ as an input, and reveals the value of $\omega$ in different $i\in I$ one by one. 
At each step, which coordinate will be revealed next depends on the values of $\omega$ revealed so far. The algorithm stops as soon as the value of $f$ is the same no matter the values of $\omega$ on the remaining coordinates. 
Define the functions $\delta_i(\mathsf T)$ and $\infl_{i}(f)$, which are respectively called the {\em revealment} and the {\em influence} of the $i$-th coordinate, by
\begin{align*}\delta_i(\mathsf T) &:=  \pi[\text{$\mathsf T$ reveals the value of } \omega_i],\\
\infl_i(f) &:= \pi [\,  f (\omega) \neq f(\tilde\omega)]\,,\end{align*}
where $\tilde\omega$ denotes the random element in $\Omega^I$ which is the same as $\omega$ in every coordinate except the $i$-th coordinate which is \emph{resampled independently}. 
\begin{theorem}[\cite{OSSS}] \label{thm:OSSS2}
For every function $f: \Omega \rightarrow \{0,1\}$, and every algorithm $\mathsf T$ determining $f$, 
\begin{equation} \label{eq:OSSS}\tag{OSSS}
\var_\pi  (f) \leq \sum_{i\in I} \delta_{i}(\mathsf T) \, \infl_{i}(f),
\end{equation}
where $\var_\pi$ is the variance with respect to the measure $\pi$.
\end{theorem}

This inequality is used as follows. First, we write Poisson-Boolean percolation as a product space. Second, we exhibit an algorithm for the event $\{0\leftrightarrow\partial\mathsf B_r \}$ for which we control the revealments. Then, we use the assumption $\lambda > \widetilde\lambda_c$ to connect the influences of the product space to the derivative of $\theta_r$. Altogether, these steps lead to \eqref{eq:mlem}.

\paragraph{Organization of the article} The next section contains some preliminaries. Section~\ref{sec:3} contains the proof of Theorem~\ref{thm:main} while the last section contains the proofs of Theorems~\ref{thm:2} and \ref{thm:3}.

\section{Background}
We introduce some notation and recall three properties of the Poisson-Boolean percolation that we will need in the proofs of the next sections. 
\paragraph{Further notation} For $x\in \bbZ^d$, introduce the squared box $\mathsf S^{x}:= x+[-1/2,1/2)^d$ around~$x$. In order to apply the OSSS inequality, we wish to write our probability space as a product space. To do this, we introduce the following notation. For any integer $n\ge1$ and $x\in \bbZ^d$, let $$\eta_{(x,n)} := \eta\cap \big(\mathsf S^x\times [n-1,n)\big),$$ which corresponds to all the balls of $\eta$ centered at a point in $\mathsf S^x$ with radius in $[n-1,n)$. All the constants $c_i$ below (in particular in the lemmata) are independent of all the parameters.

\paragraph{Insertion tolerance} We will need the following insertion tolerance property. Consider ${\bf r_*}$ and ${\bf r^*}$ such that 
\begin{equation}\label{eq:radass}\tag{IT}
\bbP_\lambda[\calD_x]:=c_{\rm IT}=c_{\rm IT}(\lambda)>0,
\end{equation}
where $\calD_x$ is the event that there exists $(z,R)\in \eta$ with $z\in\mathsf S^x$ and
$\mathsf B_{\bf r_*}^x\subset \mathsf B_R^z\subset \mathsf B_{\bf r^*}^x$.
Without loss of generality (the radius distribution may be scaled by a constant factor), we further assume  that ${\bf r_*}$ and ${\bf r^*}$ satisfy the following conditions (these will be useful at different stages of the proof):
\begin{equation}
\label{eq:rass}
1+2\sqrt d\le {\bf r_*}\le{\bf r^*}\le 2{\bf r_*}-2\sqrt d.
\end{equation}
While the quantity $c_{\rm IT}$ varies with $\lambda$, the dependency will be continuous and therefore irrelevant for our arguments. We will omit to refer to this subtlety in the proofs to avoid confusion.
\paragraph{FKG inequality} An {\em increasing} event $A$ is an event such that $\eta\in A$ and $\eta\subset\eta'$ implies $\eta'\in A$. The FKG inequality for Poisson point processes  states that for every $\lambda>0$ and every two increasing events $A$ and $B$,
\begin{equation}\label{eq:FKG}\tag{FKG}
\bbP_\lambda[A\cap B]\ge \bbP_\lambda[A]\bbP_\lambda[B].
\end{equation}
\paragraph{Russo's formula}
For $x\in \bbZ^d$ and an increasing event $A$, define the random variable
\begin{equation}\label{eq:pivdef}
  \textrm{Piv}_{x,A}  (\eta) :=  \mathbf 1_{\eta \notin A}\int_{\mathsf S^x} \int_{\bbR_+} \mathbf 1_{\eta \cup (z,  r) \in A} \,dz \, \mu(d r).
\end{equation}
Russo's formula yields that
\begin{equation}\label{eq:russo}\tag{Russo}
\frac{d}{d\lambda}\bbP_\lambda [A] = \sum_{\substack{x \in \bbZ^d}} \bbE_\lambda [ \textrm{Piv}_{x,A} ].
\end{equation}

\section{Proof of Lemma~\ref{lem:mlem}}\label{sec:3}

The next subsection contains the proof of Lemma~\ref{lem:mlem} conditioned on two lemmata, which are proved in the next two subsections.
\subsection{Proof of Lemma~\ref{lem:mlem}}
As mentioned above, the proof of the lemma is obtained by applying the OSSS inequality to a truncated version of the probability space generated by the independent variables $(\eta_{(x,n)})_{x\in\bbZ^d,n\ge1}$.
In this section, we fix $L\ge 2r>0$. Set $\calA:=\{0\stackrel{}\longleftrightarrow \partial\mathsf B_r\}$ and $f={\bf 1}_\calA$. Define 
$$I_L:=\{(x,n)\in\bbZ^d\times\bbN\text{ such that } \|x\|\le L\text{ and }1\le n\le L\}.$$

Also, for $r\ge0$, let $\eta_{\bf g}$ denote the union of the $\eta_{(x,n)}$ for $(x,n)\notin I_L$. For $i$ being either $(x,n)$ or ${\bf g}$, set
$\Omega_i$ to be the space of possible $\eta_i$ and $\pi_i$ the law of $\eta_i$ under $\bbP_\lambda$. 

For $0\le s\le r$, apply the OSSS inequality (Theorem~\ref{thm:OSSS2}) to
\begin{itemize}[noitemsep,nolistsep]
\item the product space $(\prod_{i\in I}\Omega_i,\otimes_{i\in I}\pi_i)$ where 
$I:=I_L\cup\{{\bf g}\}$,
\item the indicator function $f=\mathbf 1_\calA$  considered as a function from $\prod_{i\in I}\Omega_i$ onto $\{0,1\}$, 
\item  the algorithm $\mathsf T_{s,L}$ defined below.
\end{itemize}
\begin{definition}[Algorithm $\mathsf T_{s,L}$]
  Fix a deterministic ordering of $I$. Set $i_0 = {\bf g}$, and reveal $\eta_{{\bf g}}$. At each step $t$, assume that $\{i_0,\dots,i_{t-1}\} \subset I$ has been revealed. Then,\begin{itemize}
\item 
If there exists $(x,n) \in I\setminus \{i_0,\dots,i_{t-1}\}$ such that 
 the Euclidean distance between $\mathsf S^x$ and the connected components of $\partial\mathsf B_s$ in $\calO(\eta_{i_1}\cup\dots\cup\eta_{i_{t-1}})$ is smaller than $n$,
then set $i_{t+1}$ to be the smallest such $(x,n)$ for some fixed ordering. 
\item If such an $(x,n)$ does not exist, halt the algorithm.\end{itemize}
\end{definition}
\begin{remark}Roughly speaking, the algorithm checks $\calO(\eta_{\bf g})$ and then discovers the connected components of $\partial\mathsf B_s$.
\end{remark}
Theorem \ref{thm:OSSS2} implies that 
\begin{equation}\label{eq:OSSSa}
\theta_{r}(1-\theta_{r}) \leq 2\sum_{i\in I} \delta_{i}(\mathsf T_{s,L}) \infl_i(f).
\end{equation}
\medbreak
By construction, the random variable $\eta_{\bf g}$ is automatically revealed so  its revealment is 1. Also, 
\begin{equation}\label{eq:r1}\infl_{\bf g}(f)\le \bbP_\lambda[\exists (z,R)\in\eta\text{ with $R\ge L$ and $\mathsf B_R^z\cap \mathsf B_r\ne \emptyset$}]\end{equation}
so that this quantity tends to 0 as $L$ tends to infinity (thanks to the moment assumption on $\mu$). 

Let us now bound the revealment for $(x,n)\in I_L$. The random variable $\eta_{(x,n)}$ is revealed by $\mathsf T_{s,L}$ if the Euclidean distance between $\mathsf S^x$ and the connected component of $\partial \mathsf B_s$ is smaller than $n$. Let $\mathsf S^x_n$ be the union of the boxes $\mathsf S^y$ that contain a point at distance exactly $n$ from $\mathsf S^x$. We deduce that
\begin{equation}\label{eq:r2}\delta_{(x,n)}(\mathsf T_{s,L})\le \bbP_\lambda[\mathsf S^x_n\stackrel{}{\longleftrightarrow}\partial\mathsf B_s].\end{equation}
Overall,  plugging the previous bounds \eqref{eq:r1} and \eqref{eq:r2} on the revealments of the algorithms $\mathsf T_{s,L}$ into \eqref{eq:OSSSa}, we obtain
\begin{equation*}
  \theta_{r}(1-\theta_{r}) \le  2\sum_{(x,n)\in I_L}\bbP_\lambda[\mathsf S^x_n\stackrel{}{\longleftrightarrow}\partial\mathsf B_s]\cdot \infl_{(x,n)}(f) +o(1).
\end{equation*}
(Above, $o(1)$ denotes a quantity tending to $0$ as $L$ tends to infinity.)
Letting $L$ tend to infinity (note that the influence $\infl_{(x,n)}$ does not depend on the algorithms $\mathsf T_{s,L}$), we find
\begin{align}\label{eq:op}\theta_{r}(1-\theta_{r}) &\le 2\sum_{\substack{x\in \bbZ^d\\ n\ge1}}\bbP_\lambda[\mathsf S^x_n\stackrel{}{\longleftrightarrow}\partial\mathsf B_s]\cdot \infl_{(x,n)}(f).
\end{align}
In order to conclude the proof, we need the following two lemmata. First, a simple union bound argument allows us to bound $\bbP_\lambda[\mathsf S^x_n\stackrel{}{\longleftrightarrow}\partial\mathsf B_s]$ in terms of the one-arm probability. More precisely, consider the following lemma, which will be proved at the end of the section.

\begin{lemma}\label{lem:1}
  Fix $\lambda_0>0$. There exists $c_2>0$ such that for every $\lambda\ge\lambda_0$, $x\in\bbZ^d$ and $n\ge1$,
  \begin{equation*}\int_0^r\bbP_\lambda[\mathsf S^x_n\stackrel{}{\longleftrightarrow}\partial\mathsf B_s]ds\le c_2 n^{d-1}\Sigma_r(\lambda).\end{equation*}
\end{lemma}
Integrating \eqref{eq:op} for radii between 0 and $r$ and using the lemma above gives
\begin{equation}
  \label{eq:4}
  r\theta_{r}(1-\theta_{r}) \le 2c_2\Sigma_r\  \sum_{\substack{x\in \bbZ^d\\ n\ge1}} n^{d-1}\infl_{(x,n)}(f).
\end{equation}
The most delicate step is to relate the influences in the equation above with the pivotal probabilities appearing in the derivative formula \eqref{eq:russo}. This is the content of the following lemma, which will be proved in Section~\ref{sec:lemma}.
\begin{lemma} \label{lem:infcomp}

There exists $c_3>0$ such that for every $x\in \bbZ^d$, every $n\ge1$ and every $\lambda>\widetilde\lambda_c$, 
\begin{equation*}
\sum_{x \in \bbZ^d} {\rm Inf}_{(x,n)}(f)\le c_3\, n^{4d-2}\, \mu[n-1,n]\sum_{\substack{x\in\bbZ^d}} \bbE_\lambda [ \emph{Piv}_{x,A} ] .
\end{equation*}
\end{lemma}
Dividing~\eqref{eq:4} by $\Sigma_r$ and applying the lemma above  gives
\begin{align*}
\frac r{\Sigma_r}\theta_{r}(1-\theta_{r}) \leq \,&2c_2c_3\sum_{n\ge1} n^{5d-3}\mu[n-1,n] \sum_{x\in\bbZ^d}\bbE_\lambda [ \textrm{Piv}_{x,A}] \\
\le \,&c_1\sum_{x\in\bbZ^d}\bbE_\lambda [ \textrm{Piv}_{x,A} ] \stackrel{\eqref{eq:russo}}{=}\, c_1 \theta'_r.
\end{align*}
This implies Lemma~\ref{lem:mlem}. Before proving lemmata~\ref{lem:1} and \ref{lem:infcomp}, we would like to make several remarks concerning the proof above.

\begin{remark}
The interest of working with $\mathsf T_{s,L}$ is to have a finite number of coordinates for the algorithm. We could have stated an OSSS inequality valid for countably many states and get \eqref{eq:op} directly, but we believe the previous strategy to be shorter and thriftier.
\end{remark}

\begin{remark}
Lemma~\ref{lem:1} may a priori be improved. Indeed, the union bound in the first inequality is quite wasteful. Nonetheless, with the moment assumption on the radius distribution, the claim above is sufficient.
\end{remark}

\begin{remark}
Lemma~\ref{lem:infcomp} is slightly too strong. We could replace $n^{-2d-1}$ by any sequence $a_n$ such $a_n n^{2d-1}$ is summable.  If Lemma~\ref{lem:1} was improved by replacing $n^{d-1}$ by a sequence $b_n$ going to infinity more slowly, then the condition on $a_n$ would become that $n^da_nb_n $ is summable. This is an observation to keep in mind if one wants to improve the moment assumption. Here, we favored the simplest proof possible and did not try to optimize the two lemmata.
\end{remark}

\begin{remark}
We refer to $\lambda\ge \lambda_0$ in Lemma~\ref{lem:1} instead of $\lambda>\widetilde\lambda_c$ (even though the lemma is anyway used in this context) to illustrate that the only place where $\lambda>\widetilde\lambda_c$ is used is in Lemma~\ref{lem:infcomp}.
\end{remark}

We finish this section with the proof of Lemma~\ref{lem:1}.

\begin{proof}[Proof of Lemma~\ref{lem:1}]
Let $Y$ be the subset of $\bbZ^d$ such that 
$\displaystyle\mathsf S^x_n=\bigcup_{y\in Y}\mathsf S^y.$
We have that 
\begin{align*}\bbP_\lambda[\mathsf S^x_n\stackrel{}{\longleftrightarrow}\partial\mathsf B_s]&\le \sum_{y\in Y}\bbP_\lambda[\mathsf S^y\stackrel{}{\longleftrightarrow}\partial\mathsf B_s]\le \frac1{c_{\rm IT}}\sum_{y\in Y}\bbP_\lambda[y\stackrel{}{\longleftrightarrow}\partial\mathsf B_s].
\end{align*}
where in the last inequality, we used the FKG inequality, \eqref{eq:radass} and the fact that if $\mathsf S^y\stackrel{}{\longleftrightarrow}\partial\mathsf B_s$ and the event $\mathcal D_y$ defined below \eqref{eq:radass} occur, then $y$ is connected to $\partial\mathsf B_s$. 

Integrating on $s$ between $0$ and $r$ and observing that $y$ is at distance $|s-\|y\||$ of $\partial\mathsf B_s$, we deduce that
$$\int_0^r\bbP_\lambda[y\stackrel{}{\longleftrightarrow}\partial\mathsf B_s]ds\le\int_0^r\theta_{|s-\|y\||}(\lambda)ds\le 2\Sigma_r(\lambda).
 $$
Since the cardinality of $Y$  is bounded by a constant times $n^{d-1}$, the result follows. 
\end{proof}
\subsection{A technical statement regarding connection probabilities above $\widetilde\lambda_c$}

The following lemma will be instrumental in the proof of Lemma~\ref{lem:infcomp}. It is the unique place where we use the assumption $\lambda>\widetilde\lambda_c$. 

Below, $X\stackrel{Z}{\longleftrightarrow}Y$ means that $X$ is connected to $Y$ in $\calO(\eta^Z)$, where $\eta^Z$ is the set of $(z,R)\in \eta$ such that $\mathsf B_z(R)\subset Z$. 
We highlight that this is not the same as the existence of a continuous path in $\calO(\eta)\cap Z$ connecting $x$ and $y$, since such a path could a priori pass through regions in $Z$ which are only covered by balls intersecting $\bbR^d\setminus Z$.

\begin{lemma}\label{lem:tech}
There exists a constant $c_4>0$ such that for every $\lambda > \widetilde\lambda_c $ and $r\ge {\bf r^*}$, 
\begin{equation}\label{eq:connect}
\bbP_\lambda \big[ 0 \stackrel{\mathsf B_r}{\longleftrightarrow} \mathsf B_{\bf r^*}^x \big]  \geq \frac{c_4}{r^{2d - 2 }}\qquad\text{for every } x\in\partial\mathsf B_r.
\end{equation}
\end{lemma}
\begin{remark}Before diving into the proof, let us first explain how we will use the assumption $\lambda>\widetilde\lambda_c$. Fix $r>0$. If $\mathsf B_r$ is connected to $ \partial\mathsf B_{2r}$, then there must exist $x\in\bbZ^d$ with $\mathsf S^x\cap\partial\mathsf B_r\ne \emptyset$ such that $\mathsf S^x$ is connected to $\partial\mathsf B_{2r}$.  
Since above $\widetilde\lambda_c$, the probability of the former is bounded away from 0 uniformly in $r$, and since 
the probabilities of the latter events are all smaller than $\bbP_\lambda[\mathsf S^0\longleftrightarrow \partial\mathsf B_r]$, the union bound gives
\begin{equation}\label{eq:a}\bbP_\lambda[\mathsf S^0\longleftrightarrow \partial\mathsf B_r]\ge \frac{c_5}{r^{d-1}}.\end{equation}\end{remark}

The argument will rely on the following observation. As explained in \eqref{eq:a}, the probability that $\mathsf S^0$ is connected to $\partial\mathsf B_r$ does not decay quickly. This event implies the existence of $y$ such that $\mathsf S^0$ is connected in $\mathsf B_r$ to $\mathsf B_{\bf r_*}^y$, and one ball $\mathsf B_R^z$ (with $(z,R)\in\eta$) covering $\mathsf B_{\bf r_*}^y$ and intersecting the complement of $\mathsf B_R^z$. The problem is that this site $y$ may be quite far from $\partial\mathsf B_r$. Nonetheless, this seems unlikely since the cost (by the moment assumption on $\mu$) of having a large ball $\mathsf B_R^z$ intersecting $\mathsf B_{{\bf r}_*}^y$ is overwhelmed by the probability that the latter is connected to $\partial\mathsf B_r$ in $\mathsf B_{r-\|y\|}^y$ by a path in $\calO(\eta)$ using a priori smaller balls (and maybe even only balls included in $\mathsf B_r$).
The proof below will harvest this idea though with some important variations due to the fact that all the balls under consideration must remain inside $\mathsf B_r$. One of the key idea is the introduction of a new scale ${\bf r}^{**}$ and an induction on the probability that $\mathsf S^0$ is connected to $\mathsf B_{{\bf r}^{**}}^x$ in $\mathsf B_r$, where $x\in\partial\mathsf B_r$.

 \begin{proof}
We will prove that there exist ${\bf r^{**}}>0$ and $c_6>0$ such that 
\begin{equation}\label{eq:aux2}
\bbP_\lambda \big[ \mathsf S^0 \stackrel{\mathsf B_r}{\longleftrightarrow} \mathsf B_{{\bf r^{**}}}^x \big]  \geq \frac{c_6}{r^{2d - 2 }}\qquad\text{for every }x\in\mathsf \partial B_r.
\end{equation}
This is amply sufficient to prove the lemma since, if $T$ denotes the set of $t\in\bbZ^d$ such that $\mathsf S^t\cap \mathsf B_{2{\bf r^{**}}}^x\ne \emptyset$ and $\mathsf B^z_{\bf r^*}\subset \mathsf B_r$, then 
$$\bbP_\lambda \big[0 \stackrel{\mathsf B_r}{\longleftrightarrow} \mathsf B^x_{\bf r^*} \big]\stackrel{\rm (FKG)}\ge \bbP_\lambda[\calD_0]\big(\prod_{t\in T}\bbP_\lambda[\calD_t]\big)\bbP_\lambda \big[ \mathsf S^0 \stackrel{\mathsf B_r}{\longleftrightarrow} \mathsf B_{{\bf r^{**}}}^x \big]\stackrel{\eqref{eq:radass}}\ge {c_{\rm IT}}^{|T|+1}\cdot \frac{c_6}{r^{2d - 2 }}.$$
We therefore focus on the proof of \eqref{eq:aux2}. We will fix ${\bf r^{**}}\ge 2 {\bf r^*}$ sufficiently large (see the end of the proof). Introduce $Y:=\bbZ^d\cap \mathsf B_{r-{\bf r^{**}}}$ and a finite set $Z \subset \partial \mathsf B_r$ such that the union of the balls $\mathsf B_{\bf r_*}^{y}$ and $
\mathsf B_{{\bf r^{**}}}^{z}$ with $y \in Y$ and $z \in Z$ cover the ball $\mathsf B_r$. Note that if $0$ is connected to $\partial\mathsf B_r$, 
then either one of the $z\in Z$ is such that $0$ is connected to $\mathsf B_{{\bf r^{**}}}^z$ in $\mathsf B_r$, or there exists $y\in
Y$ such that the event 
$$
\calA(y) := \big\lbrace \mathsf S^0 \stackrel{\mathsf B_r}{\longleftrightarrow} \mathsf B_{\bf r_*}^{y}  \big \rbrace \cap 
\big\lbrace \exists(u,R)\in\eta \text{ such that $\mathsf B_R^u$ intersects both $\mathsf B_{\bf r_*}^{y}$ and $\partial\mathsf B_r$}\big\rbrace
$$
occurs.
The union bound therefore implies that
\begin{equation}\label{eq:1}\sum_{z\in Z}\bbP_\lambda[\mathsf S^0 \stackrel{\mathsf B_r}{\longleftrightarrow} \mathsf B_{{\bf r^{**}}}^z]+\sum_{\substack{y\in Y}}\bbP_\lambda[\calA(y)]\ge \bbP_\lambda[\mathsf S^0\longleftrightarrow \partial\mathsf B_r]\stackrel{\eqref{eq:a}}\ge \frac{c_5}{r^{d-1}}.\end{equation}
For $y\in Y$, introduce  $x:=(r/\|y\|)y\in\partial\mathsf B_r$ and $\overline r=r-\|y\|$.  The event on the right of the definition of $\calA(y)$ is independent of the event on the left. Since any point in $\mathsf B_{\bf r_*}^{y}$ is at a distance at least $\overline r-{\bf r_*}$ of $\partial\mathsf B_r$, we deduce that 
$$\bbP_\lambda[\calA(y)]\le \Big(c_7\int_{\overline r-{\bf r_*}}^\infty r^{d-1}\mu(dr)\Big)\cdot\bbP_\lambda[\mathsf S^0 \stackrel{\mathsf B_r}{\longleftrightarrow} \mathsf B_{\bf r_*}^{y}]\le \frac{c_8}{\overline r^{\,4d-2}}\,\bbP_\lambda[\mathsf S^0 \stackrel{\mathsf B_r}{\longleftrightarrow} \mathsf B_{\bf r_*}^{y}].$$
In the second inequality, we used  the moment assumption on $\mu$ and the fact that $\overline r-{\bf r_*}\ge \overline r/2$ since ${\bf r^{**}}\ge2{\bf r^*}$. 
Using this latter assumption one more time
 implies that
\begin{align}\nonumber \bbP_\lambda[\mathsf S^0 \stackrel{\mathsf B_r}{\longleftrightarrow} \mathsf B_{{\bf r^{**}}}^x]&\stackrel{\rm (FKG)}\ge \bbP_\lambda[\mathsf S^0 \stackrel{\mathsf B_r}{\longleftrightarrow} \mathsf B_{\bf r_*}^{y}]\cdot \bbP_\lambda[\mathcal D_y] \cdot \bbP_\lambda[\mathsf S^y \stackrel{\mathsf B_{\overline r}^y}{\longleftrightarrow} \mathsf B_{{\bf r^{**}}}^x]\\ 
\label{eq:3}&\stackrel{\phantom{\rm (FKG)}}\ge \frac{{c_{\rm IT}}\, \overline r^{\,4d-2}}{c_8}\cdot\bbP_\lambda[\calA(y)]\cdot \bbP_\lambda[\mathsf S^y \stackrel{\mathsf B_{\overline r}^y}{\longleftrightarrow} \mathsf B_{{\bf r^{**}}}^x].
\end{align}
From now on, define $U(r):= \bbP_\lambda [ \mathsf S^0 \stackrel{\mathsf B_r}{\longleftrightarrow} \mathsf B_{{\bf r^{**}}}^x ]$ for $x\in\partial\mathsf B_r$ (the choice of $x$ is irrelevant by invariance under the rotations). 
Now, one may choose $Z$ in such a way that $|Z|\le c_{9}r^{d-1}$. Plugging \eqref{eq:3} in \eqref{eq:1}, we obtain the following inequality
$$ c_{9}r^{d-1}\,U(r)+c_{10}\sum_{\substack{y\in Y}}\frac{U(r)}{U(r-\|y\|)(r-\|y\|)^{4d-2}}\ge \frac{c_5}{r^{d-1}}.$$
Using that $U(r)=1$ for every $r\leq {\bf r^{**}}$, we deduce \eqref{eq:aux2} by induction, provided ${\bf r^{**}}$ is chosen large enough to start with.
\end{proof}

\subsection{Proof of Lemma~\ref{lem:infcomp}}\label{sec:lemma}
We are now in a position to prove Lemma~\ref{lem:infcomp}. Fix $x \in \bbZ^d$ and $r,n\ge1$. Define 
$$\calP_x(n):=\{0\longleftrightarrow \mathsf B_{n}^x\}\cap\{\mathsf B_n^x \longleftrightarrow \partial \mathsf B_r\}\cap\{0 \longleftrightarrow \partial \mathsf B_r\}^c.$$
Introduce ${\bf C}={\bf C}(\eta) := \{z \in \bbR^d: z \stackrel{}{\longleftrightarrow} 0\}$ and ${\bf C}'={\bf C}'(\eta):=\{z \in\bbR^d: z \stackrel{}{\longleftrightarrow} \partial\mathsf B_r\}$ the connected components of 0 and $\partial\mathsf B_r$ in $\calO(\eta)$.  
 Our first goal is to prove that, conditionally on $\calP_x(n)$, the probability that ${\bf C}$ and ${\bf C}'$ are close to each other in $\mathsf B^x_{3n}$ is not too small.  
 \paragraph{Claim.} {\em We have that  }
\begin{equation*}\label{eq:bb}\bbP_\lambda[\calP_x(n)\text{ and }d({\bf C}\cap\mathsf B^x_{3n},{\bf C}')<{\bf r^{*}}]\ge \frac{c_{11}}{n^{3d-2}}\cdot\bbP_\lambda[\calP_x(n)].\end{equation*}
\bigbreak
\noindent Before proving this claim, let us conclude the proof of the lemma. If the event on the left occurs, there must exist $y\in \bbZ^d$ within a distance at most  ${\bf r_*}$ of both ${\bf C}\cap\mathsf B^x_{3n}$ and ${\bf C}'$: simply pick a point $z\in\bbR^d$ at a distance smaller than ${\bf r^*}/2$ of both ${\bf C}\cap\mathsf B^x_{3n}$ and ${\bf C'}$, and then take $y\in\bbZ^d$ within distance $\sqrt d$ of it (we used the inequality on the right in \eqref{eq:rass}). 
In particular, $\calP_y({\bf r_*})$ must occur for this $y$ and we deduce that 
\begin{equation}\label{eq:cccc}\bbP_\lambda[\calP_x(n)]\le c_{12}\,n^{3d-2}\,\sum_{\substack{y\in\bbZ^d\\ \|y-x\|\le 3n+{\bf r_*}}}\bbP_\lambda[\calP_y({\bf r_*})].\end{equation}
To conclude, observe that the definition of ${\bf r^*}$ given by \eqref{eq:radass} implies that  for all $y$,
\begin{equation}\label{eq:e}{c_{\rm IT}}\,\bbP_\lambda[\calP_y({\bf r_*})]\le \bbE_\lambda[{\rm Piv}_{y,A}].\end{equation}
Now, one easily gets that 
 \begin{equation}\label{eq:d}
 \infl_{(x,n)}(f) \le \lambda \mu[n-1,n]\times \bbP_\lambda[ \calP_x(n+\sqrt d)].
 \end{equation}
(Simply observe that $\calP_x(n+\sqrt d)$ must occur in $\eta$, and that the resampled Poisson point process $\tilde\eta_{(x,n)}$ must contain at least one point). Plugging \eqref{eq:cccc} (applied to $n+\sqrt d$) in \eqref{eq:d}, then using \eqref{eq:e}, and finally summing over every $x\in\bbZ^d$ gives the lemma. 

\begin{proof}[Proof of the claim] 
The proof consists in expressing the probability that ${\bf C}$ and ${\bf C'}$ come within a distance ${\bf r^*}$ of each other in terms of the probability that they remain at a distance at least ${\bf r^*}$ of each other.

 Fix $y\in\bbZ^d$.
For a subset $C$ of $\bbR^d$, let $u_C=u_C(y)$ be the point of $\mathsf S^y$ furthest to $C$, and $v_C=v_C(y)$ the point of $C$ closest to $u_C$. Consider the event
\begin{align*}
\calE_y&:=\{{\bf C}\cap\mathsf B^x_{n}\ne\emptyset\text{ and }d(u_{\bf C},{\bf C})\ge{\bf r^*}\}.
\end{align*}
Note that the event $\calE_y$ is measurable in terms of ${\bf C}$. 
Let us study,  on $\calE_y$, the conditional expectation with respect to ${\bf C}$. Introduce the three events
\begin{align*}
\calF_y:=\calD_{u_{\bf C}}
\qquad\calG_y:=\{u_{\bf C}\stackrel{\bbR^d\setminus{\bf C}}{\longleftrightarrow}\mathsf B_{v_{\bf C}}({\bf r^*})\}
\qquad&\calH_y:=\{\mathsf B^x_{n}\cap\mathsf S^y\stackrel{\bbR^d\setminus{\bf C}}{\longleftrightarrow}\partial\mathsf B_r\}.
\end{align*}

On the event $\calE_y$, conditioned on ${\bf C}$, $\eta^{\bbR^d\setminus{\bf C}}$ has the same law as $\tilde\eta^{\bbR^d\setminus {\bf C}}$ for some independent realization $\tilde\eta$ of the Poisson point process (recall the definition of $\eta^Z$ from the previous section). Also, the distance  between $u_{\bf C}$ and $v_{\bf C}$ (or equivalently ${\bf C}$) is larger than ${\bf r^*}$. Therefore, we may use \eqref{eq:radass}, Lemma~\ref{lem:tech} and the FKG inequality to get the following inequality between conditional expectations: 
\begin{align}\label{eq:r}
\bbP_\lambda[\calF_y\cap\calG_y\cap\calH_y|{\bf C}]&\ge \bbP_\lambda[\calF_y|{\bf C}]\cdot \bbP_\lambda[\calG_y|{\bf C}]\cdot \bbP_\lambda[\calH_y|{\bf C}]\ge  {c_{\rm IT}}\,\frac{c_4}{n^{2d-2}}\bbP_\lambda[\calH_y|{\bf C}]\quad\text{a.s. on $\calE$.}
\end{align}

Now, observe that if $\calE_y$ and $\calH_y$ occur simultaneously, then $\calP_x(n)$ occurs (we use ${\bf r^*}\ge \sqrt d$). Furthermore if $\calF_y$ and $\calG_y$ also occur simultaneously with them, then $\mathsf B^{v_{\bf C}}_{\bf r^*}\cap{\bf C}'\ne \emptyset$. Since by construction $v_{\bf C}\in \mathsf B^x_{3n}$, we deduce that $d({\bf C}\cap\mathsf B^x_{3n},{\bf C}')<{\bf r^*}$. Integrating \eqref{eq:r} on $\calE_y$, we deduce that 
\begin{align*}
\bbP_\lambda[\calP_x(n)\text{ and }d({\bf C}\cap\mathsf B^x_{3n},{\bf C}')<{\bf r^*}]&\ge\bbP_\lambda[\calE_y\cap\calF_y\cap\calG_y\cap\calH_y]\ge \frac{{c_{\rm IT}}\, c_4}{n^{2d-2}}\,\bbP_\lambda[\calE_y\cap\calH_y].
\end{align*}
Observe that if $\calP_x(n)$ and $d({\bf C}\cap\mathsf B^x_{n},{\bf C}')\ge {\bf r^*}$ occur, then there exists $y\in \bbZ^d$ such that the event on the right-hand side occurs. In particular, $\mathsf S^y$ must intersect $\mathsf B^x_{n}$ for $\calH_y$ to occur, so that there are $c_6n^d$ possible values for $y$. Summing over all these values, we therefore get that 
\begin{align*}c_6 n^d\,\bbP_\lambda[\calP_x(n)\text{ and }d({\bf C}\cap\mathsf B^x_{3n},{\bf C}')< {\bf r^*}]&\ge \frac{{c_{\rm IT}}\, c_4}{n^{2d-2}}\,\bbP_\lambda[\calP_x(n)\text{ and }d({\bf C}\cap\mathsf B^x_{n},{\bf C}')\ge {\bf r^*}],\end{align*}
which implies the claim readily.
\end{proof}
\section{Renormalization of crossing probabilities}
\label{sec:renorm-ofcr-prob}
In this section, we prove Theorems~\ref{thm:2} and \ref{thm:3}. The proof is quite different in the regime where $\mu[r,\infty]$ decays exponentially fast (light tail), and in the regime where it does not (heavy tail).  Also, since in the heavy tail regime the renormalization argument performed below may be of some use in the study of other models, or for different distributions $\mu$, we prove a quantitative lemma which we believe to be of independent interest.

In this section, we reboot the count for the constants $c_i$ starting from $c_1$. The section is organized as follows. In Section~\ref{sec:4.1}, we prove a renormalization lemma which will enable us to derive the theorems. In Sections~\ref{sec:4.2} and \ref{sec:4.3}, we derive Theorem~\ref{thm:2} and Theorem~\ref{thm:3} respectively.

\subsection{The renormalization lemma}\label{sec:4.1}

Introduce, for every $\delta,\alpha,\lambda\ge0$ and $r\ge1$, the two functions
\begin{align*}
\pi_r^\delta(\lambda)&:=\bbP_\lambda[\exists (z,R)\in\eta\text{ such that } \mathsf B_R^z \cap \mathsf B_{2\delta r}\ne \emptyset\text{ and }\mathsf B_R^z \cap \partial \mathsf B_{(1-2\delta)r}\ne \emptyset],\\
\theta_r^\alpha(\lambda) &:=  \bbP_{\lambda} [\mathsf B_{\alpha r} \longleftrightarrow \partial \mathsf B_r].
\end{align*}
Note that $\pi_r^0(\lambda)=\phi_r(\lambda)$ and $\theta_r^0(\lambda)=\theta_r(\lambda)$ by definition. 

\begin{remark}
The quantity $\pi_r^\delta(\lambda)$ is expressed in terms of $\mu$ as follows: let $c_d$ be the area of the sphere of radius 1 in $\bbR^d$, then\begin{equation}\label{eq:p1}
\pi_r^\delta(\lambda)=1-\exp\Big(-\lambda c_d \int_0^\infty a^{d-1}\mu[|a-2\delta r|\vee |r-2\delta r-a|,\infty]da\Big).
\end{equation}
\end{remark}

 The following lemma will be the key to the proofs of the theorems.
\begin{lemma}[Renormalization inequality]
\label{lem:17}
For every $0<\alpha\le \delta\le 1/4$, there exists $c_1>0$ such that for every $\lambda,r>0$, \begin{equation}
  \label{eq:14}
  \theta^\alpha_r (\lambda)\le \pi_r^\delta(\lambda)+\frac{c_1}{(\delta^2\alpha)^d}\,\max_{\substack{u,v\ge\delta \\ u+v=1-\alpha}}\theta^\alpha_{ur}(\lambda)\,\theta^\alpha_{vr}(\lambda).
\end{equation}
\end{lemma}

Note that the smaller the $\alpha$ and $\delta$, the larger the entropic factor $c_1/(\delta^2\alpha)^d$. 
We will see that the choices of $\alpha$ and $\delta$ are important in the applications. Except for the exponential bound on $\theta_r(\lambda)$ in the proof of Theorem~\ref{thm:2}, we will always pick $\delta=\alpha$.
%
\begin{proof}
Fix $0<\alpha\le \delta\le 1/4$ and $\lambda,r>0$. Set 
$$E:=\{2\delta r\}\cup ([2\delta r,(1-2\delta)r]\cap \tfrac{\alpha\delta}{2} r\bbZ)\cup\{(1-2\delta)r\}$$
and index the elements of $E$ in the increasing order by $r_0< r_1<\dots< r_\ell$. 
 Recall the notation $\eta^Z$ for the set of $(z,R)\in \eta$ with $\mathsf B_R^z\subset Z$, and introduce the three events 
\begin{align*}\calA&:=\{\exists (z,R)\in\eta:\mathsf B_R^z\cap\mathsf B_{2\delta r}\ne\emptyset\text{ and }\mathsf B_R^z\cap\partial\mathsf B_{(1-2\delta) r}\ne\emptyset\},\\
\calB_k&:=\{\mathsf B_{\alpha r}\longleftrightarrow\partial\mathsf B_{r_k}\text{ in }\calO(\eta^{\mathsf B_{r_{k+1}}})\},\\
\calC_k&:=\{\mathsf B_{r_{k+1}}\longleftrightarrow\partial\mathsf B_{r}\}.
\end{align*}
If $\mathsf B_{\alpha r}$ is connected to $\partial\mathsf B_r$ but $\calA$ does not occur, then $\calB_k\cap\calC_k$ must occur for some $0\le k<\ell$ (we use that $\delta\ge\alpha$). Furthermore, by construction, $\calB_k$ and $\calC_k$ are independent since $\calC_k$ is measurable with respect to  $\calO(\eta\setminus\eta^{\mathsf B_{r_{k}}})$. These two observations together imply that 
\begin{align}\label{eq:j1}
\theta^\alpha_r (\lambda)&\le \bbP_\lambda[\calA]+\sum_{k=0}^{\ell-1}\bbP_\lambda[\calB_k\cap\calC_k]=\pi^\delta_r(\lambda)+\sum_{k=0}^{\ell-1}\bbP_\lambda[\calB_k]\bbP_\lambda[\calC_k].
\end{align}
Fix $0\le k < \ell$. Let $X$ denote a set of points $x\in\partial\mathsf B_{\alpha r}$ such that the union of the balls $\mathsf B_{\alpha \delta r}^x$ for $x\in X$ covers $\partial\mathsf B_{\alpha r}$. Choose $X$ such that $|X|\le c_2 \delta^{-(d-1)}$. Applying the union bound, we find
\begin{align} \bbP_\lambda[\calB_k]&\le \sum_{x\in X}\bbP_\lambda[\mathsf B_{\alpha \delta r}^x\longleftrightarrow \partial\mathsf B_{r_k-\alpha r}^x]\le|X|\,\theta_{r_k-\alpha r}^\alpha(\lambda)\le \frac {c_2}{\delta^{d-1}}\,\theta_{r_k-\alpha r}^\alpha(\lambda).\label{eq:12}
\end{align}
Similarly, using a covering of   $\partial\mathsf B_{r_{k+1}}$ with at most $c_2'(\alpha\delta)^{1-d}$ balls of radius $\alpha\delta r$ centered on $\partial \mathsf B_{r_k}$, we  obtain
\begin{equation}
  \label{eq:11}
  \bbP_\lambda[\calC_k]\le \frac{c_2'}{(\delta\alpha)^{d-1}} \,\theta_{r-r_k}^\alpha(\lambda).
\end{equation}
Plugging \eqref{eq:12} and \eqref{eq:11} in \eqref{eq:j1} and using the fact that $\ell\le 2/(\alpha\delta)$ concludes the proof.
\end{proof}
The previous lemma leads to the following useful corollary.
\begin{corollary}\label{cor:a}
For $\lambda<\widetilde\lambda_c$,
$\displaystyle\lim_{r\rightarrow\infty}\bbP_\lambda[\mathsf B_r\longleftrightarrow\partial\mathsf B_{2r}]=0.$
\end{corollary}

\begin{proof}
Choose $\alpha=1/6$. In this case 
$$\theta_{r}^\alpha(\lambda)=\bbP_\lambda[\mathsf B_{r/6}\longleftrightarrow\partial\mathsf B_{r}].$$
Using a covering of $\partial\mathsf B_r$ with balls of radius $B_{r/6}$, it suffices to prove that $\theta_{r}^\alpha(\lambda)$ converges to $0$.  
 A reasoning similar to the bound on $\bbP_\lambda[\calB_k]$ in the last proof implies that for every $s\in[\alpha r, r]$, 
\begin{equation}\theta_r^\alpha(\lambda)\le c_3\,\theta_s^\alpha(\lambda),\label{eq:l1}\end{equation}
which enables us to rewrite \eqref{eq:14} (applied with $\delta=\alpha=1/6$) as
$$\theta_r^\alpha(\lambda)\le \pi^{\alpha}_r(\lambda)+c_4 \,\theta^\alpha_{\alpha r}(\lambda)^2.$$
Now, fix $\ep>0$ such that $2c_4\ep<1$ and work with $r_0$ large enough that $\pi^{\alpha}_r(\lambda)\le \ep/2$ for every $r\ge r_0$. Since $\lambda<\widetilde \lambda_c$, we can pick $r\ge r_0$ such that $\theta_r^\alpha(\lambda)<\ep$. We deduce inductively that $$\theta_{r/\alpha^k}^\alpha(\lambda)\le \ep$$ for every $k\ge1$.
Using \eqref{eq:l1} one last time gives that $\theta_s^\alpha(\lambda)\le c_3\ep$ for every $s\ge r$.
\end{proof}
\subsection{Proof of Theorem~\ref{thm:2}}\label{sec:4.2}

Consider $\lambda<\widetilde \lambda_c$. Fix $\delta\in(0,\tfrac12)$ and observe that since $\mu[r,\infty]\le \exp(-cr)$, \eqref{eq:p1} implies that there exists $c'>0$ such that $\pi_r^\delta(\lambda)\le \exp(-c' r)$ for every $r\ge1$.
We now proceed in two steps: we first show that $\theta_r(\lambda)\le r^{-d-1}$ for $r$ large enough, and then improve this estimate to an exponential decay.  

\paragraph{Polynomial bound on $\theta_r(\lambda)$}
Lemma~\ref{lem:17} applied to $\alpha=\delta=1/6$ implies that for every $r\ge 1$
\begin{equation}
  \label{eq:6}
\theta_r^\alpha(\lambda) \le  e^{-c' r}+c_5 \max_{\substack{u,v\ge\alpha \\ u+v=1-\alpha}}\theta_{ur}^\alpha(\lambda)\theta_{vr}^\alpha(\lambda).
\end{equation}
Since $\lambda<\widetilde \lambda_c$, Corollary~\ref{cor:a} implies that $\theta_r^\alpha(\lambda)$ converges to $0$ as $r$ tends to infinity. In particular, we deduce that for every $\ep>0$,  
$$\theta_r^\alpha(\lambda)\le e^{-c' r}+\ep\,\theta_{\alpha r}^\alpha(\lambda)$$
for $r$ large enough. A simple induction implies that $\theta_r^\alpha(\lambda)\le r^{-d-1}$ for $r$ large enough. 
\paragraph{Exponential bound on $\theta_r(\lambda)$}
Choose $\delta=1/6$. For each $r\ge6$, apply Lemma~\ref{lem:17} with $\alpha=\alpha(r):=1/r$ to get that 
$$\theta_r(\lambda)\le e^{-c'r}+c_6r^d \max_{\substack{u,v\ge\tfrac16 \\ u+v=1-\alpha(r)}}\theta_{ur}(\lambda)\theta_{vr}(\lambda),
$$
where we use that  
$$\theta_r(\lambda)\le \theta_r^{\alpha(r)}(\lambda)=\bbP_\lambda[\mathsf B_1\longleftrightarrow\partial\mathsf B_r]\le \tfrac1{c_{\rm IT}}\theta_r(\lambda).$$ 
Set  $c_7:=2 \, c_6 \, e $ and consider $r_0$ large enough so that for every $r\ge \delta r_0$,
\begin{align*}
&c_7r^d\,e^{-c'r}\le \tfrac12e^{-r/r_0},\text{ and }\\
&c_7r^d\,\theta_r(\lambda)\le e^{-1}.\end{align*}
(The second constraint is satisfied thanks to the  polynomial bound derived in the first part of the proof.)
By induction on $k$, one can show that for every $k\ge 0$,    
$$\forall r\in[\delta r_0, r_0/(1-\delta)^k]\quad c_7r^d\,\theta_r(\lambda)\le \exp(-r/r_0).$$

\subsection{Proof of Theorem~\ref{thm:3}}\label{sec:4.3}

In this section, we prove Theorem~\ref{thm:3} by proving the following quantitative lemma. We believe the lemma to be of value for distributions $\mu$ other than the one considered in Theorem~\ref{thm:3}. Before that, note that the lower bound
\begin{equation}\theta_r(\lambda)\ge \phi_r(\lambda)\label{eq:m1}\end{equation}
for every $r>0$ is trivial since 0 is connected to $\partial\mathsf B_r$ in the case where there exists $(z,R)\in \eta$ such that $\mathsf B_R^z$ contains 0 and intersects $\partial\mathsf B_r$. Therefore, the main concern of this section will be an upper bound on $\theta_r^\alpha(\lambda)$ (which is larger than $\theta_r(\lambda)$).
 \begin{lemma}\label{thm:16}
Consider $\lambda>0$, $\eta<1$ and $\alpha\in(0,\tfrac14)$ small enough  such that $\alpha^{\eta}+(1-\alpha)^{\eta}\ge 1$. For every $\ep>0$ sufficiently small and $r\ge1$, if there exists $r_0\le \alpha r$ such that
\begin{align}
\theta_s^\alpha(\lambda)&\le \ep\, \pi_{r}^\alpha(\lambda)^{(s/r)^\eta} &\forall s\in[r_0,r_0/\alpha]\label{eq:f1}\\
\pi_s^\alpha(\lambda)&\le \ep\,\pi_r^\alpha(\lambda)^{(s/r)^\eta}&\forall s\in[r_0,(1-\alpha)r],\label{eq:f2}
\end{align}
then
\begin{equation}
\theta_{r}^\alpha(\lambda)\le (1+\ep)\,\pi_{r}^\alpha(\lambda).
\end{equation}

\end{lemma}

\begin{remark}\label{rem:22} 1) Note that distributions $\mu$ with exponential decay do not satisfy the assumptions of the lemma for any $r$, provided that $\ep$ is small enough. \medbreak
\noindent
2) For distributions $\mu$ satisfying $\lim_{r\rightarrow\infty}\pi_r^\alpha(\lambda)^{1/r^\eta}= \kappa$, for some constant $\kappa$, the existence of $r_0$ satisfying \eqref{eq:f1} when $\lambda<\tilde\lambda_c$ follows directly from
\begin{itemize}[noitemsep,nolistsep]
\item picking $r_0$ such that $\theta_r^\alpha(\lambda)<\tfrac{\ep}{2\kappa}$ for every $r\ge r_0$ (by Corollary~\ref{cor:a}),
\item picking $r$ large enough that $\pi_r^\alpha(\lambda)^{1/r^\eta}\ge \kappa 2^{-(\alpha/r_0)^\eta}$. 
\end{itemize}
\medbreak\noindent
3) The second assumption \eqref{eq:f2} is a regularity statement on $\pi_r^\alpha(\lambda)$ that can be obtained from the regularity of $\mu[r,\infty]$ using \eqref{eq:p1}.\end{remark}

The proof of the lemma will be based on the recursive relation given by Lemma~\ref{lem:17}. 
We use a strategy which is inspired by the study of differential equations. We control the function $f(s):=\theta_s^\alpha(\lambda)$  in terms of $g(s):=\pi_r^\alpha(\lambda)^{(s/r)^\eta}$  uniformly for $s\in [ r_0, (1-\alpha) r]$ to finally deduce a bound for $f(r)$ using again Lemma~\ref{lem:17}.

\begin{proof}
Set $C=c_1\alpha^{-2d}$, where $c_1$ is given by Lemma~\ref{lem:17}, and assume that $\ep<\tfrac1{4C}$. The key ingredient will be the following inequality: for every $s\in[r_0,(1-\alpha)r]$, 
\begin{equation}\label{eq:h1}
\theta_s^\alpha(\lambda)\le 2\ep\, \pi_r^\alpha(\lambda)^{(s/r)^\eta}.
\end{equation}
For $s\in[r_0,r_0/\alpha]$, \eqref{eq:h1} follows from \eqref{eq:f1}. To obtain the claim for $s> r_0/\alpha$, use the induction hypothesis in the second line below to get
\begin{align*}\theta^\alpha_s(\lambda)&\stackrel{{\eqref{eq:14}}}\le \pi_s^\alpha(\lambda)~+~C \max_{\substack{u,v\ge \alpha \\ u+v=1-\alpha}}\theta^\alpha_{us}(\lambda)\,\theta^\alpha_{vs}(\lambda)\\
&\stackrel{\eqref{eq:h1}}\le \pi_s^\alpha(\lambda)~+~4C\ep^2 \max_{\substack{u,v\ge\alpha \\ u+v=1-\alpha}}\pi^\alpha_r(\lambda)^{(u^\eta+v^\eta)(s/r)^\eta}\\
&\stackrel{\eqref{eq:f2}}\le (\ep+4C\ep^2)\, \pi_r^\alpha(\lambda)^{(s/r)^\eta}\\
 &\stackrel{\phantom{\eqref{eq:f2}}}\le 2\ep\,\pi_r^\alpha(\lambda)^{(s/r)^\eta}.\end{align*}
In the first line, we applied \eqref{eq:14} with $\delta=\alpha$. In the second line, we used \eqref{eq:h1} for $us,vs\in[r_0,s-r_0]$ and in the third line we used $u^\eta+v^\eta\ge 1$  for every  $u,v\ge\alpha$ satisfying $u+v=1-\alpha$. 
One may apply the first two lines in the previous sequence of inequalities for $s=r$ (since in such case $ur,vr\in[r_0,(1-\alpha)r]$) so that
\begin{align*}\theta^\alpha_r(\lambda)&
\le \pi_r^\alpha(\lambda)+4C\ep^2 \max_{\substack{u,v\ge\alpha \\ u+v=1-\alpha}}\pi^\alpha_r(\lambda)^{u^\eta+v^\eta}\le (1+\ep)\, \pi_r^\alpha(\lambda).\end{align*}
This concludes the proof.
\end{proof}


\begin{proof}[Proof of Theorem~\ref{thm:3}] Fix $\lambda<\widetilde\lambda_c$.
Since \eqref{eq:m1} gives the lower bound, we focus on the upper bound. Under the condition C1, define $f(\alpha,r)$ by   
\begin{equation}
\pi_r^\alpha(\lambda)=f(\alpha,r) r^{-c},\label{eq:n1}
\end{equation}
 and, under the condition C2, define $g(\alpha,r)$ by 
\begin{equation}
\pi_r^\alpha(\lambda)=g(\alpha,r) r^{d-\gamma} \exp[-c (\frac r 2-{2\alpha} r)^\gamma]. \label{eq:n2}
\end{equation}
Using \eqref{eq:p1}, one can check that $f(\alpha,r)$ and $g(\alpha,r)$ converge (as $r$ tends to infinity) uniformly in $0\le \alpha<\tfrac14$ to two continuous positive functions $f(\alpha)$ and $g(\alpha)$.

Consider $\ep$ very small and use Corollary~\ref{cor:a} to guarantee that $\theta_{s}^\alpha(\lambda)<\ep$ for every $s$ large enough. Then, for every distribution $\mu$ satisfying C1 or C2, one can find $0<\eta<1$, $r_0$ large enough such that \eqref{eq:f1} and \eqref{eq:f2} are satisfied for every  $r\ge r_0$ large enough, and $\alpha$ small enough  (use \eqref{eq:n1} and \eqref{eq:n2}  and Remark~\ref{rem:22} to check the two conditions), so that Lemma~\ref{thm:16} implies that for $r$ large enough,
\begin{equation}\theta_r^\alpha(\lambda)\le (1+\ep)\pi_r^\alpha(\lambda)\label{eq:bn}.\end{equation}

Thanks to \eqref{eq:n1}, we obtain the claim in case C1 by letting $\ep$ and then $\alpha$ tend to 0. 
For case C2, we need to do more, since letting $\alpha$ tend to 0 slowly gives only that $\theta_r(\lambda)\le \phi_r(\lambda)^{1+o(1)}$. 
More precisely, we prove the following claim. 

\paragraph{Claim}{\em Assume that $\mu$ satisfies $\mathrm{C2}$. Fix $\overline \gamma\in(\gamma,1)$. There exist $\alpha>0$ and $R<\infty$ such that for every $k\ge0$ and every $r\ge 3^kR$, we have that 
\begin{equation}\label{eq:b2}\theta_{r}^{\alpha_k}(\lambda)\le (1+2^{-k})\pi_r^{\alpha_k}(\lambda)\end{equation}
where
$\alpha_k:= \alpha 3^{-k \overline\gamma}$.}
\medbreak
 Before proving the claim, note that it implies, together with \eqref{eq:n2}, that for every $k\ge0$ and $r\in [3^k R,3^{k+1}R)$,
 $$\theta_r(\lambda)\le\theta_{r}^{\alpha_k}(\lambda)\le  (1+2^{-k})\pi_r^{\alpha_k}(\lambda)\le (1+2^{-k}) \frac{g(\alpha_k,r)}{g(0,r)}e^{Cr^{\gamma-\overline\gamma}} \phi_r(\lambda)$$
for some constant $C=C(\alpha,R)>0$. This proves the theorem in case C2 thanks to the choice of $\overline\gamma>\gamma$  and the uniform convergence of $g(.\,,r)$ towards a continuous function.

 \medbreak\noindent
{\em Proof of the Claim.} We prove \eqref{eq:b2} recursively. Fix $0< \alpha\le 1/8$ such that 
\begin{align}
  &2 (1-12\alpha)^\gamma \left(\tfrac{1-\alpha}2\right)^\gamma\ge 1+\tfrac14 \alpha^\gamma,\text{ and} \label{eq:9}\\
  \label{eq:10}
  &\forall x\in[0,\alpha] \quad  (1 - 12x)^{2\gamma}+(x/2)^\gamma
  \ge 1+\tfrac14  x^\gamma.
\end{align}
 Fix also $R\ge1$. There will be several conditions on $R$ that will be added during the proof. The important feature is that every time a new condition is added, one only needs to take $R$ possibly larger than before.  

Define $M$ such that for every $x\le \tfrac14$, $\tfrac2M\le g(x)\le\tfrac M2$. By uniform convergence, $R$ can be chosen large enough that for every $x\le \tfrac14$ and $r\ge R$,
\begin{equation}\tfrac1M\le g(x,r)\le M.\label{eq:b1}\end{equation}

Use \eqref{eq:bn} to choose $R$ in such a way that $\theta_r^{\alpha}(\lambda)\le 2\pi_r^\alpha(\lambda)$ for every $r\ge \alpha R$. Then, \eqref{eq:b2} follows for $k=0$ by definition. We now assume that \eqref{eq:b2} is valid for $k-1$ and prove it for $k$.
Lemma~\ref{lem:17} applied to $r\ge R2^k$ and $\alpha_k$  gives that
\begin{align}\label{eq:c1}\theta_{r}^{\alpha_k}(\lambda)&\le \pi_{r}^{\alpha_k}(\lambda)+\frac{c_1}{\alpha_k^{3d}} \,\theta_{ur}^{\alpha_k}(\lambda)\,\theta_{vr}^{\alpha_k}(\lambda),\end{align}
for some $u,v \ge\alpha_k$ such that $u+v=1-\alpha_k$. Without loss of generality, we can assume that $u \geq \frac 1 3$ and $v\le \frac{1-\alpha_k}2$. In particular, we have $ur\ge R3^{k-1}$. Also, 
$vr\ge \alpha R3^{k (1-\overline\gamma)}\ge \alpha R$. The induction hypothesis (in fact we simply bound $2^{-{(k-1)}}$ by $1$) implies that 
$$\theta_{ur}^{\alpha_k}(\lambda)\le \theta_{ur}^{\alpha_{k-1}}(\lambda)\le 2\pi_{ur}^{\alpha_{k-1}}(\lambda)\qquad \text{and}\qquad\theta_{vr}^{\alpha_k}(\lambda)\le \theta_{vr}^{\alpha}(\lambda)\le 2\pi_{vr}^\alpha(\lambda).$$
We may use the induction hypothesis to get that
\begin{align}
\theta_{ur}^{\alpha_k}(\lambda)\theta_{vr}^{\alpha_k}(\lambda)&\stackrel{\eqref{eq:b2}}\le 4 \pi_{ur}^{\alpha_{k-1}}(\lambda)\pi_{vr}^\alpha(\lambda)\nonumber\\
& \stackrel{\eqref{eq:b1}}\leq 4 M^2 (uvr^2)^{d-\gamma} \exp \big[-c(r/2)^\gamma \big((u - 4\alpha_{k-1}u)^\gamma+(v  -4 \alpha v)^\gamma\big)\big] \label{eq:8}.
\end{align}
We now bound the term $$h(v):=(u - 4\alpha_{k-1}u)^\gamma+(v  -4 \alpha v)^\gamma{=} (1 - 4\cdot 3^{\overline\gamma}\alpha_{k})^\gamma(1-\alpha_k-v)^\gamma+(1-4 \alpha)^\gamma v^\gamma$$ appearing in the exponential in \eqref{eq:8}. Since $\alpha_k\le v\le \frac{1-\alpha_k}2$, an elementary analysis of the function $h$ shows that $h(v) \ge \min (h(\alpha_k),h(\tfrac{1-\alpha_k}2))$. Using 
\eqref{eq:9} and \eqref{eq:10} to bound $h(\frac{1-\alpha_k}2)$ and $h(\alpha_k)$ respectively, we obtain 
\begin{align*}
   h(v)\ge 1+ \tfrac14 \alpha_k^\gamma \ge(1 - 4\alpha_k)^\gamma +\tfrac14  \alpha_k^\gamma.
\end{align*}
Plugging this estimate in \eqref{eq:8} and using \eqref{eq:b1} one last time, we finally get 
\begin{align*}
  \theta_{ur}^{\alpha_k}(\lambda)\theta_{vr}^{\alpha_k}(\lambda)&  \stackrel{\phantom{\eqref{eq:b3}}}\leq 4M^3 r^{2d} \exp [-\tfrac c8 (\alpha_kr)^\gamma ] \, \pi_r^{\alpha_k} (\lambda)\\
&\stackrel{\eqref{eq:b3}}\le \frac{\alpha_k^{3d}}{c_12^k}\,\pi_r^{\alpha_k}(\lambda),
\end{align*}
 which concludes the proof using \eqref{eq:c1}.
In the second line, we used that $R$ is chosen so large that for every $k\ge0$ and $r\ge 3^k R$,
\begin{align}\label{eq:b3}
\frac {c_1 2^k}{\alpha_k^{3d}} 4 M^3 r^{2d} \exp [-\tfrac c8(\alpha_kr)^\gamma] &\leq  \frac {c_1}{\alpha^{3d}} 4 M^3 r^{6d} \exp [-\tfrac c8 (\alpha r^{1-\overline\gamma})^\gamma ] \leq 1.
\end{align}
\end{proof}
\begin{remark}
By adapting the reasoning above, one can prove similar statements for other distributions $\mu$ having sub-exponential tails, for instance $\mu[r,\infty]$ decaying like $\exp[-c(\log r)^\gamma]$ or $\exp[-cr/(\log r)^\gamma]$.
\end{remark}
\paragraph{Acknowledgments} 
This research was supported by an IDEX grant from Paris-Saclay, the ERC grand CriBLaM, and the NCCR SwissMAP.

\end{document}